\documentclass[11pt, reqno]{amsart}

\usepackage{amscd, amsmath, amssymb, color}
\usepackage[colorlinks=true,urlcolor=blue]{hyperref}

\newcommand{\Tr}[2]{\mathrm{tr}_{#1}{#2}}
\newcommand{\vol}{\mathrm{Vol}}
\newcommand{\Ric}{\mathrm{Ric}}
\newcommand{\so}{\mathfrak{so}}
\newcommand{\su}{\mathfrak{su}}
\newcommand{\SO}{\mathrm{SO}}
\newcommand{\SU}{\mathrm{SU}}
\newcommand{\dvg}{\mathrm{div}}
\newcommand{\Hess}{\mathrm{Hess}}
\newcommand{\Sym}{\mathrm{Sym}}
\newcommand{\tr}{\mathrm{tr}}
\newcommand{\N}{\mathbb{N}}
\newcommand{\R}{\mathbb{R}}
\newcommand{\cL}{\mathcal{L}}
\newcommand{\cV}{\mathcal{V}}
\newcommand{\cB}{\mathcal{B}}

\renewcommand{\leq}{\leqslant}
\renewcommand{\geq}{\geqslant}
\renewcommand{\epsilon}{\varepsilon}
\renewcommand{\S}{\mathbb{S}}

\newtheorem{theorem}{Theorem}[section]
\newtheorem{lemma}[theorem]{Lemma}
\newtheorem{proposition}[theorem]{Proposition}

\newcounter{mtheorem}
\newtheorem{mtheorem}[mtheorem]{Theorem}

\theoremstyle{definition}
\newtheorem{definition}[theorem]{Definition}
\newtheorem{rk}[theorem]{Remark}

\numberwithin{equation}{section}

\begin{document}

\title{The renormalized volume of a \\ $4$-dimensional Ricci-flat ALE space}
\author{Olivier Biquard}
\address{Sorbonne Universit{\'e} \& Universit\'e de Paris, CNRS, IMJ-PRG, 75005 Paris, France}
\email{olivier.biquard@sorbonne-universite.fr}
\author{Hans-Joachim Hein}
\address{Mathematisches Institut, WWU M\"unster, 48149 M\"unster, Germany\newline\hspace*{9pt}
Department of Mathematics, Fordham University, Bronx, NY 10458, USA}
\email{hhein@uni-muenster.de}

\begin{abstract}
We introduce a natural definition of the renormalized volume of a $4$-dimensional Ricci-flat ALE space. We then prove that the renormalized volume is always less or equal than zero, with equality if and only if the ALE space is isometric to its asymptotic cone. Currently the only known examples of $4$-dimensional Ricci-flat ALE spaces are Kronheimer's gravitational instantons and their quotients, which are also known to be the only possible examples of special holonomy. We calculate the renormalized volume of these spaces in terms of Kronheimer's period map.
\end{abstract}

\maketitle
\markboth{The renormalized volume of a $4$-dimensional Ricci-flat ALE space}{Olivier Biquard and Hans-Joachim Hein}
\thispagestyle{empty}

\section{Introduction}
\label{sec:introduction}

This article grew out of an attempt to understand $4$-dimensional Ricci-flat asymptotically locally Euclidean (ALE) manifolds. The standard example of such a space is the Eguchi-Hanson metric on $T^*\S^2$ \cite{EguHan78}. The Eguchi-Hanson metric is actually hyper-K\"ahler (see Calabi \cite{Cal79}), and a classification of hyper-K\"ahler ALE $4$-manifolds was given by Kronheimer \cite{Kro89a,Kro89b}. Finite free quotients of Kronheimer spaces as classified by {\c S}uvaina \cite{Suv12} and Wright \cite{Wri12} are also examples of Ricci-flat ALE 4-manifolds. A fundamental open question due to Bando-Kasue-Nakajima \cite{BanKasNak89} asks whether there exist any other examples. Using Witten's proof of the positive mass theorem \cite{Wit81}, Nakajima \cite{Nak90} showed that such examples can never be spin. See Lock-Viaclovsky \cite{LocVia16} for some additional restrictions based on the Hitchin-Thorpe inequality of \cite{Nak90}. We did not solve this problem but we have found an interesting new volume property of $4$-dimensional Ricci-flat ALE spaces, which we will now explain.

Let $(M,g)$ be a $4$-dimensional Ricci-flat ALE manifold, or a $4$-dimensional Ricci-flat ALE orbifold with at worst finitely many isolated singularities (we shall call such an orbifold an \emph{ALE space}). By \cite{BanKasNak89}, this means that there exist a finite subgroup $\Gamma$ of $\SO(4)$ acting freely on $\S^3$ and a diffeomorphism $\Phi: (\R^4 \setminus B_1(0))/\Gamma \to M \setminus K$ for some compact subset $K\subset M$ such that for all $k \in \N_0$,
\begin{equation}\label{eq:1}
|\nabla^k_{g_0}(\Phi^*g - g_0)|_{g_0} = O(r^{-4-k})\;\,\text{as}\;\,r\to\infty,
\end{equation}
where $g_0$ denotes the Euclidean metric on $\R^4$ or on $\R^4/\Gamma$. By  \cite{CheNab15, CheTia94}, this behavior already follows from the much weaker assumption that  $(M,g)$ is a $4$-dimensional complete Ricci-flat manifold (or a complete Ricci-flat orbifold with finite singular set) of maximal volume growth at infinity.

Assume that $\Gamma \neq \{1\}$. By \cite{ChoEicVol17} there exist $K \subset M$ compact and a number $\rho_0 > 0$ such that $M \setminus K$ is uniquely foliated by hypersurfaces $\Sigma_\rho$ $(\rho>\rho_0)$ of constant mean curvature $3/\rho$ such that $\Sigma_\rho$ is a normal graph of height $O(\rho^{-3})$ over $\Phi(\partial B_\rho(0)/\Gamma)$ for any diffeomorphism $\Phi$ as in \eqref{eq:1}.

Given these preliminaries, our first main result may be stated as follows.

\begin{mtheorem}\label{th:max-vol}
Let $(M,g)$ be a $4$-dimensional Ricci-flat {\rm ALE} space with $\Gamma\neq \{1\}$. 
  
{\rm (1)} Let $\Omega_\rho \subset M$ denote the domain interior to $\Sigma_\rho$. Then the quantity
\begin{equation*}\label{eq:2}
\vol_g(\Omega_\rho) - \vol_{g_0}(B_\rho(0)/\Gamma)
\end{equation*}
has a finite limit as $\rho \to \infty$. We refer to this limit as the renormalized volume $\cV$ of $(M,g)$.

{\rm (2)} The renormalized volume satisfies $\cV \leq 0$, with equality if and only if $(M,g)\cong(\R^4/\Gamma,g_0)$.
\end{mtheorem}

Ozuch \cite[Thm 0.7]{Ozu19} has recently used this result to prove that spherical or hyperbolic $4$-orbifolds cannot be Gromov-Hausdorff limits of smooth Einstein $4$-manifolds under certain conditions.

There is a vast literature on renormalized volumes of open manifolds. Our motivation was not so much to extend these ideas to a new setting but to interpret some of our computations for Ricci-flat ALE spaces geometrically. However, as it turns out, the statement of Theorem \ref{th:max-vol} is formally similar to several known results. The examples we are aware of are Anderson's comparison theorem \cite{And01} for the classical renormalized volume \cite{Gra99} of asymptotically hyperbolic Einstein $4$-manifolds, Hu-Ji-Shi's comparison theorem \cite{HuJiShi16} for the Brendle-Chodosh relative volume \cite{BreCho14} of asymptotically hyperbolic $n$-manifolds with sharp lower scalar curvature bounds, and the identification \cite{CECY21,FanShiTam09,JauLee19} of Huisken's isoperimetric mass \cite{Hui06} of asymptotically flat $3$-manifolds of nonnegative scalar curvature with the ADM mass \cite{ADM61}, combined with the statement of the positive mass theorem \cite{SY79,Wit81}.

To prove Theorem \ref{th:max-vol}, we will first construct a special ALE diffeomorphism $\Phi$ as in \eqref{eq:1} such that $\Phi^*g - g_0$ admits an asymptotic expansion whose leading term has vanishing trace, divergence, and contraction with $\partial_r$. Due to these properties, the coordinate spheres $\Phi(\partial B_\rho(0)/\Gamma)$ are CMC to one higher order than expected. Thus, $\Sigma_\rho$ is a normal graph of height $O(\rho^{-4})$ rather than $O(\rho^{-3})$ over $\Phi(\partial B_\rho(0)/\Gamma)$, so by changing $\Phi$ very slightly we are able to create a CMC gauge without changing the leading term of the metric. The precise statement is as follows.

\begin{mtheorem}\label{th:classterms}
Let $(M,g)$ be a $4$-dimensional Ricci-flat {\rm ALE} space with $\Gamma\neq \{1\}$. Let $\{\Sigma_\rho\}_{\rho>\rho_0}$ be the canonical {\rm CMC} foliation of the end of $M$ constructed in \cite{ChoEicVol17}. Then for any $k_0 \in \N$ there exists a diffeomorphism $\Phi: (\R^4 \setminus B_1(0))/\Gamma \to M \setminus K$ of class at least $C^{k_0+1}$ such that $\Phi(\partial B_\rho(0)/\Gamma) =\Sigma_\rho$ for all $\rho > \rho_0$ and such that there exists a decomposition
\begin{equation}
\Phi^*g - g_0 = h_0  + h',
\end{equation}
where $h_0$ is either zero or comparable above and below to $r^{-4}$ in $g_0$-norm and satisfies
\begin{equation}
\cL_{r\partial_r}h_0 = -2h_0,\;\, \partial_r \,\lrcorner\,h_0 = 0, \;\, \tr_{g_0}h_0 = 0, \;\, \dvg_{g_0}h_0 = 0,\;\, \Delta_{g_0}h_0 = 0,\label{eq:redkronchar}
\end{equation}
and where the remainder $h'$ can be estimated by
\begin{equation}
\sum_{k=0}^{k_0} r^{k}|\nabla_{g_0}^k h'|_{g_0} = O_{k_0}(r^{-5})\;\,\text{as}\;\,r \to \infty.\label{eq:leading}
\end{equation}
\end{mtheorem}

The space of all symmetric $2$-tensors $h_0$ satisfying \eqref{eq:redkronchar} on $\R^4\setminus \{0\}$ is isomorphic to $\mathbf{S}^4_+ \oplus \mathbf{S}^4_-$ as a representation of $\SO(4)$, where $\mathbf{S}^\ell_\pm$ denotes the two $(\ell+1)$-dimensional fundamental representations of $\mathrm{Spin}(4)$. Each element of $\mathbf{S}^4_+$ arises as the leading term $h_0$ of a suitable Kronheimer gravitational instanton. All of this will be explained in Section \ref{sec:asympt-ricci-flat}. Here we only note that the property $\tr_{g_0}h_0 = 0$ allows us to conclude that the renormalized volume $\mathcal{V}$ exists, which proves Theorem \ref{th:max-vol}(1).

The second key step is to construct a function $u$ on $M$ with $\Delta u = 8$ and $\Phi^*u = r^2 + o(1)$, and to integrate by parts in the formula $\dvg(\Hess_0u) = \Ric\,d u = 0$. Theorem \ref{th:max-vol}(2) follows from an analysis of the boundary terms, relying on the properties $\partial_r \,\lrcorner\,h_0 = 0$, $\tr_{g_0}h_0 = 0$, and $\dvg_{g_0}h_0 = 0$.

It is worth pointing out that the proof of Theorem \ref{th:max-vol} goes through on every $4$-dimensional ALE space with $\Gamma \neq \{1\}$ and ${\rm Ric} \geq 0$ that satisfies the conclusion of Theorem \ref{th:classterms}. The property $\Delta_{g_0}h_0 = 0$ is actually not needed for this, although in the Ricci-flat case this property is automatic and allows us to \emph{prove} the crucial properties $\partial_r \,\lrcorner\,h_0 = 0$, $\tr_{g_0}h_0 = 0$, and $\dvg_{g_0}h_0 = 0$ because harmonic tensors on $\R^4 \setminus \{0\}$ are easy to classify. Thus, Theorem \ref{th:max-vol} certainly still holds under the weaker assumption that $\Ric \geq 0$ globally and $\Ric = O(r^{-6-\epsilon})$ at infinity. However, it could be that these two conditions together imply that $\Ric = 0$ globally. Using \cite{Nak90}, one can show that this is true if $M$ is spin.

Let us also point out that, in the Ricci-flat case, the leading term $h_0$ of Theorem \ref{th:classterms} can actually be zero even if $(M,g)$ is not isometric to $(\R^4/\Gamma,g_0)$. Indeed, using a computation of Auvray \cite{Auv18}, we will see that $h_0$ vanishes on a Kronheimer space if and only if the preimages in $H^2_c(M)$ of the three K\"ahler classes in $H^2(M)$ have intersection matrix $-\mu \cdot {\rm Id}$ for some $\mu\geq 0$ (see Remark \ref{rk:summary}). Thanks to Kronheimer's Torelli theorem \cite{Kro89b}, this is possible with $\mu \neq 0$ if and only if $b^2(M) \geq 3$. 

In the same spirit, our last main result computes the renormalized volume of a Kronheimer space. Kronheimer's Torelli theorem \cite{Kro89b}, which we just mentioned, provides a bijection between the set of all ALE hyper-K\"ahler metrics up to the obvious action of ${\rm Diff}_0(M)$ on a fixed smooth manifold $M$, where $M$ is determined by $\Gamma \subset {\rm SU}(2)$,
and elements of $H^2_c(M) \otimes \mathbb{R}^3$ up to the obvious SO$(3)$-action. This bijection, which is also known as the \emph{period map}, sends a hyper-K\"ahler metric to the preimages in $H^2_c(M)$ of the three K\"ahler classes in $H^2(M)$. The vectors in the image of the period map satisfy a nondegeneracy condition, which is an open dense condition, but this condition can be dropped if we allow for degenerations of smooth ALE hyper-K\"ahler metrics on $M$ with at worst isolated orbifold singularities. There is a further natural identification of $H^2_c(M)$ together with the intersection form with a Cartan subalgebra of a certain simple Lie algebra together with the Killing form.

\begin{mtheorem}\label{th:max-vol-kron}
Let $\Gamma$ be a finite subgroup of $\SU(2)$ acting freely on $\S^3$. Let $\mathfrak{h}$ be a Cartan subalgebra of the Lie algebra associated with $\Gamma$ via the McKay correspondence. Choose any element $\zeta \in \mathfrak{h} \otimes \R^3$. Let $(M,g)$ be the unique Kronheimer gravitational instanton with period point $\zeta$ up to the obvious action of $\SO(3)$ on $\mathfrak{h} \otimes \R^3$. Then in terms of the Killing form inner product on $\mathfrak{h}$,
\begin{equation}\mathcal{V}(M,g) = - \frac{1}{6}|\zeta|^2.\end{equation}
\end{mtheorem}

The basic idea behind Theorem \ref{th:max-vol} is to determine when the scaling vector field $r\partial_r$ on $\R^4/\Gamma$ can be extended to a conformal Killing field on $(M,g)$. We conclude this paper with some remarks on the analogous question for Killing fields instead of conformal Killing fields.

\subsection*{Acknowledgments} OB received support from the ``Investissements d'Avenir'' program ANR-10-IDEX-0001-02 PSL. HJH was partially supported by NSF grant DMS-1745517. We are grateful to the referees for useful comments, and to Michael Anderson for pointing out the similarity between our renormalized volume and Huisken's isoperimetric mass.

\section{Asymptotics of Ricci-flat ALE metrics in dimension $4$}
\label{sec:asympt-ricci-flat}

Recall that the \emph{Bianchi operator} of a Riemannian manifold $(M,g)$ is defined by 
\begin{align*}
\cB_g: \Gamma(\Sym^2T^*M) \to \Gamma(T^*M),\;\,
h \mapsto \dvg_g(h - \frac{1}{2}(\tr_g h)g).
\end{align*}
Unless stated otherwise, let $(M,g)$ be a $4$-dimensional Ricci-flat ALE space with $\Gamma \neq \{1\}$. 

\subsection{Preparation of an asymptotic expansion} 
\label{sec:exp}

Fix an arbitrary diffeomorphism $\Phi$ as in \eqref{eq:1}. By standard Fredholm theory arguments in weighted H\"older spaces, it is possible to modify $\Phi$ to satisfy the Bianchi gauge condition relative to $g$ and $g_0$, meaning that the tensor $$h := \Phi^*g - g_0$$ satisfies $\cB_{g_0}h = 0$, although a priori $\Phi$ will then only be of class $C^{k_0+1}$ for some $k_0 \in \N$, \eqref{eq:1} will only hold up to $k_0$ derivatives, and we may need to relax the decay rate in \eqref{eq:1} slightly. However, we do have the freedom of prescribing an arbitrarily high cutoff $k_0$. Moreover, since the Ricci-flat equation is elliptic in any Bianchi gauge, $h$ actually admits an asymptotic expansion. 

One can be completely precise about the shape of this expansion \cite{Che18} but for us the first term will be sufficient. By a standard iteration, the first term is a $\Gamma$-equivariant harmonic function $$h_0: \R^4 \setminus B_1(0) \to \Sym^2\R^4$$ that decays at infinity. As explained in \cite[Thm 5.103]{CheTia94}, thanks to the Bianchi gauge condition and because $\Gamma \neq \{1\}$, we can assume that each component function of $h_0$ is homogeneous of degree $-4$. This then justifies the rate $-4$ chosen in \eqref{eq:1}. We also learn that the nonlinear terms of the Ricci curvature are at worst $O(r^{-10})$, so we immediately get
\begin{equation}
  h = h_0  + h_1 + h_2 + h_3 + h', \label{e:expand}
\end{equation}
where the first four terms $h_0,h_1,h_2,h_3$ are homogeneous,
\begin{align}\label{eq:homg}
\cL_{r\partial_r}h_k = (-2-k)h_k,\;\,|h_k|_{g_0} \sim r^{-4-k},
\end{align}
and satisfy the pair of equations
\begin{align}\label{e:bia-harm}
\cB_{g_0}h_k = 0, \;\, \Delta_{g_0}h_k = 0.
\end{align}
Moreover, we are free to assume that the remainder $h'$ satisfies 
\begin{align}\label{eq:remainder}
\sum_{k = 0}^{k_0} r^{k}|\nabla^k_{g_0} h'|_{g_0} \leq C(k_0,\epsilon)r^{-8+\epsilon}
\end{align}
for any given $k_0 \in \N$ and $\epsilon \in (0,1)$. For simplicity, we will from now on absorb $h_1, h_2, h_3$ into $h'$, so that \eqref{e:expand} reads $h = h_0 + h'$ and \eqref{eq:remainder} holds with $C(k_0,\epsilon)r^{-8+\epsilon}$ replaced by $C(k_0)r^{-5}$. 

\subsection{Two examples of possible leading terms}

The remainder of Section \ref{sec:asympt-ricci-flat} is dedicated to a classification of the possible leading terms $h_0$. We begin by describing two types of examples.

\subsubsection{Harmonic gauge terms}
\label{sec:harm-gauge}

Somewhat surprisingly, the Bianchi gauge condition allows for some residual gauge freedom to leading order. Indeed, recall that for any arbitrary Riemannian manifold $(M,g)$ and any vector field $X$ on $M$ we have the Bochner type formula
\begin{equation}
\cB_g(\cL_Xg) = -\nabla^*\nabla X +\Ric\,X.\label{e:bochner}
\end{equation}
Therefore, if $g$ is Ricci-flat and if $X$ is harmonic, then $h:=\cL_Xg$ is in Bianchi gauge ($\cB_gh=0$) and satisfies the linearized Ricci-flat equation. So if $X$ on $(\R^4 \setminus \{0\})/\Gamma$ satisfies
\begin{enumerate}
\item $X$ is harmonic: $\Delta_{g_0}X=0$,
\item $\cL_{r\partial_r}X=-4X$, so in particular $|X|_{g_0} \sim r^{-3}$,
\end{enumerate}
then $h_0:=\cL_Xg_0$ will solve equations \eqref{eq:homg} and \eqref{e:bia-harm} for $k = 0$, as desired.

On $\R^4 \setminus \{0\}$ the vector fields satisfying (1)--(2) are exactly the ones of the form
\begin{equation}\label{e:foo}
X = \frac 1{r^4} \ell_{ij}x^j \frac \partial{\partial x^i}
\end{equation}
for an arbitrary matrix $L=(\ell_{ij})\in\R^{4\times4}$. So the $\Gamma$-invariant matrices $L \in (\R^{4 \times 4})^\Gamma$ give rise to such a harmonic vector field $X$ on $(\R^4\setminus\{0\})/\Gamma$ and hence to a possible term $h_0=\cL_Xg_0$.

For reference, let us note that thanks to the Bianchi condition, $h_0$ is trace-free if and only if it is divergence-free, but $h_0$ may not satisfy either of these conditions. In fact,
\begin{equation}\label{e:foo1}
\tr_{g_0}h_0 = \tr_{g_0}(\cL_{X}g_0) = 2\dvg_{g_0}X = \frac{2}{r^4}(\delta_{ij} - 4\frac{x^i x^j}{r^2})\ell_{ij},
\end{equation}
which vanishes if and only if $L$ is either a multiple of the identity or skew-symmetric. A special case is $L = L_0 = -2{\rm Id}$, so that $X = X_0=\nabla_{g_0}(\frac 1{r^2})$,  which is harmonic because $\frac{1}{r^2}$ is. We then have
\begin{equation}
h_0 = \cL_{X_0}g_0 = 2\Hess_{g_0}(\frac{1}{r^2}) = 2\nabla_{g_0}(-\frac{2}{r^3}\partial_r) = - \frac 4{r^4} (g_0 - 4 dr^2),\label{eq:10}
\end{equation}
which is trace-free by inspection or because $X_0$ is divergence-free because $\frac{1}{r^2}$ is harmonic.

More abstractly, let $\mathbf{U}$ denote the space of all harmonic gauge terms $h_0 = \cL_{X}g_0$ on $\R^4\setminus \{0\}$  with $X$ as in \eqref{e:foo}. As a representation of $\SO(4)$ this obviously decomposes into irreducibles as $$\mathbf{U} = \mathbf{U}_1 \oplus \mathbf{U}_2 \oplus \mathbf{U}_3,$$
where $\mathbf{U}_1 =\R$ is spanned by the Hessian of the Green's function in \eqref{eq:10}, $\mathbf{U}_2 = \Lambda^2\R^4$ corresponds to taking $L$ to be skew-symmetric in \eqref{e:foo}, and $\mathbf{U}_3 = \Sym^2_0\R^4$ corresponds to taking $L$ to be trace-free symmetric in \eqref{e:foo}. If $h_0 \in \mathbf{U}_3$, then $\Tr{g_0}{h_0} = q_0/r^6$ for the harmonic polynomial $q_0(x) := -8\ell_{ij}x^ix^j$ (see \eqref{e:foo1}), so the trace defines an equivariant projection $\mathbf{U} \to \Sym^2_0\R^4 = \mathbf{U}_3$.

\subsubsection{Kronheimer terms}
\label{sec:kron}

 Let $\Gamma$ be a finite subgroup of $\SU(2)$ that acts freely on $\S^3$. Kronheimer's period map \cite{Kro89b} is a bijection between hyper-K\"ahler ALE metrics asymptotic to $\R^4/\Gamma$ and triples
$$(\zeta_1, \zeta_2, \zeta_3) \in \mathfrak{h} \otimes \R^3$$ up to the obvious $\SO(3)$-action, where $\mathfrak{h}$ is a Cartan subalgebra of the Lie algebra associated with $\Gamma$ via the McKay correspondence. As explained in \cite[Thm 2.1]{Auv18}, Kronheimer constructed a special gauge with respect to which the tensor $-(2\pi^2/|\Gamma|)r^6 h_0$ takes the form
\begin{equation}
\begin{split}
\label{e:kron}
|\zeta_1|^2 ((rdr)^2 + \alpha_1^2 - \alpha_2^2 - \alpha_3^2)
+ |\zeta_2|^2 ((rdr)^2 + \alpha_2^2 - \alpha_3^2 - \alpha_1^2)
+ |\zeta_3|^2 ((rdr)^2 + \alpha_3^2 - \alpha_1^2 - \alpha_2^2) \\
 + \, 2\langle\zeta_1,\zeta_2\rangle (\alpha_1 \cdot \alpha_2 - rdr \cdot \alpha_3)
 + 2\langle\zeta_1,\zeta_3\rangle (\alpha_1 \cdot \alpha_3 - rdr \cdot \alpha_2)
 + 2\langle\zeta_2,\zeta_3\rangle (\alpha_2 \cdot \alpha_3 - rdr \cdot \alpha_1).
\end{split}
\end{equation}
Here $\alpha_j = I_j(rdr)$, and $(I_1,I_2,I_3)$ is the standard triple of complex structures on $\R^4$ given by 
\begin{equation}\label{e:quats}
I_1(x^1,x^2,x^3,x^4)=(-x^2,x^1,-x^4,x^3),\;\,
I_2(x^1, x^2, x^3,x^4)=(-x^3,x^4,x^1,-x^2),\;\,
I_3=I_1I_2.
\end{equation}
Then obviously \eqref{eq:homg} is satisfied, and, by computation,
\begin{equation}
\label{eq:three}
\tr_{g_0}h_0 = 0, \;\, \dvg_{g_0}h_0 = 0, \;\, \Delta_{g_0}h_0 = 0,
\end{equation}
so \eqref{e:bia-harm} is satisfied as well. 

Let $\SU(2)_\pm$ denote the two canonical subgroups of $\SO(4)$, with $\su(2)_+ = \langle I_1, I_2, I_3\rangle$. Then every Kronheimer tensor $h_0$ as in \eqref{e:kron} is invariant under $\SU(2)_-$. Notice that $\Gamma \subset \SU(2)_-$. 

The following lemma clarifies the structure of the tensors \eqref{e:kron} using representation theory. Let $\mathbf{S}^\ell_{\pm}$ be the $(\ell+1)$-dimensional irreducible representation of $\SU(2)_\pm$. The irreducible representations of $\mathrm{Spin}(4)$ are given by $\mathbf{S}_+^\ell \otimes \mathbf{S}_-^m$ ($\ell,m \in \N_0$). This descends to $\SO(4)$ if and only if $\ell + m$ is even.

\begin{lemma}\label{l:kron}
Let $\mathbf{V}$ denote the space of all symmetric $2$-tensors on $\R^4 \setminus \{0\}$ satisfying \eqref{eq:three} whose component functions are $(-4)$-homogeneous. Recall the $\SO(4)$-invariant subspaces $\mathbf{U}_1, \mathbf{U}_2 \subset \mathbf{V}$ from Section \ref{sec:harm-gauge} and let $\mathbf{W} \subset \mathbf{V}$ be the invariant complement of $\mathbf{U}_1 \oplus \mathbf{U}_2$. Let 
$$\mathbf{F}: \Sym^2\R^3 \to \mathbf{V}$$
be the linear map that sends $(\langle \zeta_i, \zeta_j\rangle) \in \Sym^2\R^3$ to the Kronheimer tensor $h_0 \in \mathbf{V}$ with $-r^6h_0$ given by \eqref{e:kron}. Then
 $\mathrm{pr}_{\mathbf{W}} \circ \mathbf{F}|_{{\rm Sym}^2_0\R^3}$ is a linear isomorphism onto a subrepresentation $\mathbf{S}^4_+ \subset \mathbf{W}$.
\end{lemma}

\begin{proof}
Even though this is not logically necessary, we begin by constraining the possible irreducible representations contained in $\mathbf{V}$. By construction, if $h_0 \in \mathbf{V}$, then all components of $h_0$ take the form $q_0/r^6$, where $q_0$ is a harmonic quadratic polynomial. Thus, $\mathbf{V}$ is an invariant subspace of
\begin{equation}\label{eq:11}
(\Sym^2_0\R^4)^{\otimes 2} = (\mathbf{S}^2_+ \otimes \mathbf{S}^2_-)^{\otimes 2} = (\mathbf{S}^4_+ \oplus \mathbf{S}^2_+ \oplus \R)\otimes(\mathbf{S}^4_- \oplus \mathbf{S}^2_- \oplus \R).
\end{equation}
The subspace $\mathbf{U}_1 \subset \mathbf{V}$ corresponds to the $\R$-component in \eqref{eq:11} and is spanned by
\begin{equation}
\label{eq:12}
\frac{1}{r^6}(3(r dr)^2 - \alpha_1^2 - \alpha_2^2 - \alpha_3^2)
\end{equation}
according to \eqref{eq:10}. The subspace $\mathbf{U}_2 \subset \mathbf{V}$ corresponds to the $(\mathbf{S}^2_+ \oplus \mathbf{S}^2_-)$-component in \eqref{eq:11}. It is clear that the three tensors $\cL_{X}g_0$, where $X = I(\nabla_{g_0}(\frac{1}{r^2}))$ for  $I = I_1, I_2, I_3$, generate the $\mathbf{S}^2_+$-part of this space. Calculating these tensors, we obtain that the $\mathbf{S}^2_+$-part of $\mathbf{U}_2$ is spanned by
\begin{equation}
\label{eq:12a}
\frac{1}{r^6}(r dr)\cdot \alpha_1, \;\,\frac{1}{r^6}(r dr) \cdot \alpha_2, \;\,\frac{1}{r^6}(r dr) \cdot \alpha_3.
\end{equation}

By the definition of $\mathbf{F}$ we have for all $\zeta \in \Sym^2\R^3$ that
\begin{equation*}
\begin{split}
-r^6 \mathbf{F}(\zeta) = \zeta_{11} ((rdr)^2 + \alpha_1^2 - \alpha_2^2 - \alpha_3^2)
+ \zeta_{22} ((rdr)^2 + \alpha_2^2 - \alpha_3^2 - \alpha_1^2)
+ \zeta_{33} ((rdr)^2 + \alpha_3^2 - \alpha_1^2 - \alpha_2^2) \\
+ \, 2\zeta_{12} (\alpha_1 \cdot \alpha_2 - rdr \cdot \alpha_3)
+ 2\zeta_{13} (\alpha_1 \cdot \alpha_3 - rdr \cdot \alpha_2)
+ 2\zeta_{23} (\alpha_2 \cdot \alpha_3 - rdr \cdot \alpha_1).
\end{split}
\end{equation*}
Define a new linear map $\mathbf{G}$ on $\Sym^2\R^3$ by 
\begin{equation}\label{eq:13}
\begin{split}
-r^6 \mathbf{G}(\zeta) := \frac23(\zeta_{11} (2\alpha_1^2 - \alpha_2^2 - \alpha_3^2) + \zeta_{22} (2\alpha_2^2 - \alpha_3^2 - \alpha_1^2) + \zeta_{33} (2\alpha_3^2 - \alpha_1^2 - \alpha_2^2)) \\
+\, 2\zeta_{12} \alpha_1 \cdot \alpha_2 + 2\zeta_{13} \alpha_1 \cdot \alpha_3 + 2\zeta_{23} \alpha_2 \cdot \alpha_3.
\end{split}
\end{equation}
It is easy to check using \eqref{eq:12}, \eqref{eq:12a} that $\mathbf{F} - \mathbf{G}$ takes values in $\mathbf{U}_1 \oplus \mathbf{U}_2$. Because $(\alpha_1,\alpha_2,\alpha_3)$ is a basis of $\mathbf{S}_+^2$ and because $\mathbf{S}^4_+ = \Sym_0^2\mathbf{S}^2_+$, it is clear from \eqref{eq:13} that $\mathbf{G}$ defines an isomorphism from ${\rm Sym}^2_0\R^3$ onto the $\mathbf{S}^4_+$-component in \eqref{eq:11}, so that necessarily $\mathbf{G} = \mathrm{pr}_{\mathbf{W}} \circ \mathbf{F}$ on ${\rm Sym}^2_0\R^3$.
\end{proof}

\begin{rk}
We can transport the given $\SU(2)_+$-action on $\mathbf{S}^4_+ \subset \mathbf{V}$ to $\Sym^2_0\R^3$ by using the linear isomorphism $\mathrm{pr}_{\mathbf{W}} \circ \mathbf{F}$ of Lemma \ref{l:kron}. We obtain an irreducible representation of $\SU(2)_+$ on $\Sym^2_0\R^3$ with kernel equal to $\{\pm \mathrm{Id}_{\R^4}\}$, and hence an irreducible representation of $\SO(3)$ on $\Sym^2_0\R^3$. This is \emph{not} equal to the canonical $\SO(3)$-action on $\Sym^2_0\R^3$ although it is of course conjugate to it.
\end{rk}

\begin{rk}\label{rk:kron}
From \eqref{eq:12a} and \eqref{eq:13}, if some element $h_0 \in \mathbf{V}$ belongs to one of the two invariant subspaces $\mathbf{S}^2_+$ or $\mathbf{S}^4_+$, then $h_0(\partial_r, \partial_r) = 0$. Since $\mathrm{im}\,\mathbf{F} \subset \mathbf{U}_1 \oplus \mathbf{S}^2_+ \oplus \mathbf{S}^4_+$, it follows that 
\begin{equation*}
(\mathrm{pr}_{\mathbf{U}_1} \circ \mathbf{F})(\zeta) = \mathbf{F}(\zeta)(\partial_r, \partial_r) \cdot \frac{r^4}{6}\Hess_{g_0}(\frac{1}{r^2}) = -\frac{1}{6}\tr(\zeta) \cdot \Hess_{g_0}(\frac{1}{r^2})
\end{equation*}
for all $\zeta \in \Sym^2\R^3$. Thus, if $\zeta = (\langle \zeta_i, \zeta_j\rangle)$ with $(\zeta_1,\zeta_2,\zeta_3) \in \mathfrak{h} \otimes \R^3$, then the coefficient has a sign and vanishes if and only if the associated Kronheimer space is isometric to $\R^4/\Gamma$. This observation led us to our main theorem and is in fact equivalent to it in the case of a Kronheimer space.
\end{rk}

\subsection{Classification of the possible leading terms}
\label{sec:classterms}

Lemma \ref{l:kron} motivates the following:

\begin{definition}\label{d:red_kron}
A \emph{reduced Kronheimer term} is an element of $\mathbf{S}^4_+ = \mathrm{im}(\mathrm{pr}_\mathbf{W} \circ \mathbf{F}) \subset \mathbf{V}$. Explicitly, a reduced Kronheimer term is a symmetric $2$-tensor $h_0^+$ on $\R^4 \setminus \{0\}$ of the form
\begin{equation}\label{e:redkron}
\begin{split}
-\frac32 r^6 h_0^+ = \zeta_{11}(2\alpha_1^2 - \alpha_2^2 - \alpha_3^2) + \zeta_{22}(2\alpha_2^2 - \alpha_3^2 - \alpha_1^2) + \zeta_{33}(2\alpha_3^2 - \alpha_1^2-\alpha_2^2)\\
+\, 3\zeta_{12}(\alpha_1\cdot\alpha_2) + 3\zeta_{13}(\alpha_1 \cdot \alpha_3) + 3\zeta_{23}(\alpha_2 \cdot \alpha_3), 
\end{split}
\end{equation}
where $\alpha_j = I_j(rdr)$ for $I_1,I_2,I_3$ as in \eqref{e:quats} and where $(\zeta_{ij})$ is any symmetric $3 \times 3$ matrix.

A \emph{reduced Kronheimer term for the opposite orientation} is a symmetric $2$-tensor $h_0^-$ on $\R^4\setminus\{0\}$ of the form $h_0^-=R^*h_0^+$ for some reduced Kronheimer term $h_0^+$ and some $R\in {\rm O}(4) \setminus \SO(4)$.
\end{definition}

\begin{rk}\label{rk:summary}
Using this language, the discussion of Section \ref{sec:kron} implies that for $\Gamma \subset {\rm SU}(2)$ and $\mathfrak{h}$ a Cartan subalgebra of the simple Lie algebra associated with $\Gamma$ via the McKay correspondence, for $\zeta \in \mathfrak{h} \otimes \R^3$, and for $(M,g)$ the unique Kronheimer gravitational instanton with period point $\zeta$, there exist ALE coordinates for $(M,g)$ such that the leading term $h_0$ satisfies $h_0 = (|\Gamma|/2\pi^2) h_0^+$, where $h_0^+$ is the reduced Kronheimer term given by \eqref{e:redkron} for $\zeta_{ij} = \langle \zeta_i, \zeta_j\rangle$. In  particular, $h_0 = 0$ if and only if $\langle \zeta_i,\zeta_j \rangle = \mu \delta_{ij}$ for some $\mu \geq 0$. However, we should also recall from Remark \ref{rk:kron} that in Kronheimer's \emph{original} ALE coordinates, the leading term vanishes if and only if $(M,g) \cong (\R^4/\Gamma,g_0)$.
\end{rk}

We are now in position to classify the possible leading terms of Ricci-flat ALE metrics.

\begin{proposition}\label{p:classify}
Let $h_0$ be a symmetric $2$-tensor on $\R^4 \setminus \{0\}$ that satisfies 
\begin{align}\label{e:leading_cond}
\cL_{r\partial_r}h_0 = -2h_0,\;\,\cB_{g_0}h_0 = 0, \;\, \Delta_{g_0}h_0 = 0.
\end{align}

{\rm (1)} There exists a unique decomposition
\begin{align}\label{e:final_gauge}
h_0 = h_0^++ h_0^- +  \cL_{X_1}g_0 + \cL_{X_2} g_0+ \cL_{X_3}g_0,
\end{align}
where $h_0^\pm$ are reduced Kronheimer terms for the two orientations and where
\begin{align}\label{e:harm_gauge}
X_k(x) = \frac{1}{r^4}L_k x, \;\,L_k \in \R^{4\times 4},
\end{align}
with $L_1$ a multiple of the identity, $L_2$ skew-symmetric, and $L_3$ trace-free symmetric. All terms of the decomposition \eqref{e:final_gauge} again satisfy all the conditions of \eqref{e:leading_cond}. 

{\rm (2)} We have the following characterizations:
\begin{eqnarray}
\tr_{g_0}h_0 = 0 \Longleftrightarrow \dvg_{g_0}h_0 = 0 \Longleftrightarrow X_3 = 0,\label{e:tracefreeness}\\
\partial_r \,\lrcorner\,h_0 = 0 \Longleftrightarrow X_1 = X_2 = X_3 = 0.\label{e:tangentiality}
\end{eqnarray}

{\rm (3)} Let $G$ be any subgroup of $\SO(4)$. If $h_0$ is $G$-invariant, then so are $h_0^+$, $h_0^-$, $L_1$, $L_2$, $L_3$.
\end{proposition}

\begin{proof}
Define an equivariant linear map
\begin{align}
\label{e:def_H}
\mathbf{H}: \Sym^2\R^4 \otimes \Sym^2_0\R^4 \to \R^4 \otimes \Sym^3\R^4,\;\, h \otimes q \mapsto r^8\cB_{g_0}(\frac{q(x)}{r^6} h),
\end{align}
where $q(x) =q_{ij} x^i x^j$ denotes the harmonic quadratic polynomial associated with $q \in \Sym^2_0\R^4$. The kernel of $\mathbf{H}$ is naturally and equivariantly isomorphic to the space of all tensors $h_0$ satisfying \eqref{e:leading_cond}. By Section \ref{sec:harm-gauge}, $\mathrm{ker}\,\mathbf{H}$ contains a $16$-dimensional invariant subspace $\R \oplus (\mathbf{S}^2_+ \oplus \mathbf{S}^2_-) \oplus (\mathbf{S}^2_+ \otimes \mathbf{S}^2_-)$ consisting of tensors of the form $\cL_{X_1}g_0 + \cL_{X_2}g_0 + \cL_{X_3}g_0$ as above. By Lemma \ref{l:kron} in Section \ref{sec:kron}, $\mathrm{ker}\,\mathbf{H}$ also contains a $10$-dimensional invariant subspace $\mathbf{S}^4_+ \oplus \mathbf{S}^4_-$ consisting of reduced Kronheimer terms for the two orientations. Thus, to prove item (1) it suffices to compute ${\rm dim}\,{\rm ker}\,\mathbf{H} = 26$, which can be done using the Maxima script in Appendix \ref{app:code}.

To prove item (2), first observe that \eqref{e:tracefreeness} follows from the fact that $\cB_{g_0}h_0 = 0$, \eqref{e:foo1}, and \eqref{e:redkron}.
To prove \eqref{e:tangentiality}, notice that the map $h_0 \mapsto \partial_r \,\lrcorner\,h_0$ is equivariant, so by Schur's lemma and \eqref{e:redkron} it suffices to show that it has trivial kernel on tensors of the form $\cL_{X_1}g_0$, $\cL_{X_2}g_0$, and $\cL_{X_3}g_0$. For the first two types this is obvious from \eqref{eq:12} and \eqref{eq:12a}, respectively. For the remaining type, we write $X_3(x) = \frac{1}{r^4}Lx$ for all $x\in\R^4\setminus\{0\}$ with $L = (\ell_{ij}) \in \R^{4\times 4}$ trace-free symmetric. Then
$$(\partial_r \,\lrcorner\, \cL_{X_3}g_0)(\partial_k) = -\frac{2}{r^7}(r^2 Lx + 2q(x) x)^k,$$
where $q(x) = \ell_{ij}x^ix^j$ is the quadratic polynomial associated with $L$. If this vanishes for all $k$, then every $v \in \R^4\setminus\{0\}$ is an eigenvector of $L$ with eigenvalue $-2q(\frac{v}{|v|})$, so that $L = 0$.

For item (3), note that the projections from $\mathrm{ker}\,\mathbf{H}$ onto its $\SO(4)$-invariant subspaces are $\SO(4)$-equivariant, so if $h_0$ is $G$-invariant, then so are $h_0^+$, $h_0^-$, and $\cL_{X_k}g_0$ ($k = 1,2,3$). It remains to prove that the vector fields $X_k$, or equivalently the matrices $L_k$, are $G$-invariant. But this is clear because the map $X \mapsto \cL_Xg_0$ for $X$ of the form \eqref{e:foo} is $\SO(4)$-equivariant and injective.
\end{proof}

\begin{rk}
It follows from \eqref{e:tracefreeness} that $\mathbf{W} = \mathbf{S}^4_+ \oplus \mathbf{S}^4_-$ in Lemma \ref{l:kron} is as small as it can be.
\end{rk}

\section{The volume inequality}
\label{sec:volume-inequality}

\subsection{Proof of Theorem \ref{th:classterms}}
\label{sec:proof-B}

Let $(M,g)$ be a $4$-dimensional Ricci-flat {\rm ALE} space with $\Gamma\neq \{1\}$. Let $k_0 \in \N$. By Section \ref{sec:exp} there exists a diffeomorphism 
\begin{align*}
\Phi: (\R^4 \setminus B_1(0))/\Gamma \to M \setminus K
\end{align*}
of class at least $C^{k_0+1}$ such that 
\begin{equation*}
\Phi^*g - g_0 = h_0 + h',
\end{equation*}
where the leading term $h_0$ satisfies
\begin{equation*}
\cL_{r\partial_r}h_0 = -2h_0,\;\,\cB_{g_0}h_0 = 0, \;\, \Delta_{g_0}h_0 = 0,
\end{equation*}
and where the remainder $h'$ can be estimated by
\begin{align*}
\sum_{k = 0}^{k_0} r^{k}|\nabla^k_{g_0} h'|_{g_0} \leq C(k_0)r^{-5}.
\end{align*}
Then Proposition \ref{p:classify} tells us that after lifting through $\Gamma$, 
\begin{align*}
h_0 = h_0^++ h_0^- +  \cL_{X_1}g_0 + \cL_{X_2} g_0+ \cL_{X_3}g_0,
\end{align*}
where $h_0^\pm$ are $\Gamma$-invariant reduced Kronheimer terms for the two orientations and $X_1$, $X_2$, $X_3$ are $\Gamma$-invariant harmonic vector fields of the appropriate types. If we compose $\Phi$ with the time-$1$ flows of these vector fields, the metric changes by the corresponding Lie derivatives of $g_0$ to leading order. By Taylor expansion, since $|\nabla_{g_0}^j X_k|_{g_0} = O(r^{-3-j})$ for $k = 1,2,3$ and all $j \in \N_0$, all other changes to  the metric are certainly $O(r^{-5})$ as $r \to \infty$, with as many derivatives as the regularity of $h'$ allows. Thus, we are free to assume without loss that $X_1 = X_2 = X_3 = 0$, and hence that 
\begin{align}\label{e:improve}
\partial_r \,\lrcorner\,h_0 = 0, \;\, \tr_{g_0}h_0 = 0, \;\, \dvg_{g_0}h_0 = 0,
\end{align}
where the latter two properties are equivalent thanks to the Bianchi gauge condition.

We now modify $\Phi$ to become a CMC gauge, preserving all of its other properties except for the number of derivatives of $h'$ that we control, which will drop by a bounded amount. Since this does not affect the statement of Theorem \ref{th:classterms}, we will from now on treat $k_0$ as a generic constant.

Let us recall how the canonical {\rm CMC} foliation $\{\Sigma_\rho\}_{\rho>\rho_0}$ was constructed in \cite{ChoEicVol17}. For all $\rho>0$ we consider the homothety $h_\rho(x)=\rho x$ in $\R^4/\Gamma$. Then on any compact set $K\subset (\R^4 \setminus \{0\})/\Gamma$,
\begin{equation}
h_\rho^*\Phi^*g - g_0= O_K(\rho^{-4})\;\,\text{as}\;\,\rho\to\infty\label{eq:4}
\end{equation}
including all derivatives up to order $k_0$. The unit sphere $\S^3/\Gamma$ has constant mean curvature $3$, and the linearization of the CMC equation on normal deformations of $\S^3/\Gamma$ is the Jacobi operator 
$J$ $=$ $\Delta_{\S^3/\Gamma}+3$. On $\S^3$ the kernel of $\Delta_{\S^3}+3$ is given by the restriction to $\S^3$ of the linear functions on $\R^4$ and corresponds to translations of $\S^3$ inside $\R^4$. As $\Gamma\neq\{1\}$, this flexibility disappears in $\R^4/\Gamma$, so that $J$ is invertible. It follows that for $\rho\gg1$ we can deform $\S^3/\Gamma$ into a hypersurface $\tilde\Sigma_\rho$ of constant mean curvature $3$ with respect to $h_\rho^*\Phi^*g$. More precisely, we can write $\tilde\Sigma_\rho$ as a radial graph
\begin{equation*}
\tilde\Sigma_\rho = \{ (1 + f_\rho(x))x \in \R^4/\Gamma: x \in \S^3/\Gamma  \},
\end{equation*}
and \eqref{eq:4} implies that for all $0\leq k\leq k_0$ we have 
\begin{equation}
\nabla^kf_\rho = O(\rho^{-4})\;\,\text{as}\;\,\rho\to\infty.\label{eq:6}
\end{equation}
Pushing $\tilde\Sigma_\rho$ forward by $\Phi \circ h_\rho$ now gives the required CMC hypersurfaces $\Sigma_\rho$. The same argument applied to the obvious family of CMC hypersurfaces around $\tilde\Sigma_\rho$ with mean curvature between $\frac32$ and $6$ tells us that the family $\{\Sigma_\rho\}_{\rho>\rho_0}$ is actually a foliation.

The key point is that in our setting we are able to improve \eqref{eq:6} by one order by using \eqref{e:improve}. To see this, we will compute the mean curvature $H$ of the unit sphere $\S^3/\Gamma$ with respect to $h_\rho^*\Phi^*g$. It is clear from \eqref{eq:4} that $H = 3 + O(\rho^{-4})$ up to $k_0$ derivatives, which leads to \eqref{eq:6}. However, we can improve this to $H = 3 + O(\rho^{-5})$, with a corresponding improvement in \eqref{eq:6}, by viewing $\S^3/\Gamma$ as a level set of the function $u = r^2$ and using \eqref{e:improve} to exhibit cancellations in the formula
\begin{equation}\label{e:levelMC}
H = \frac{1}{|\nabla u|}(\Delta u - \frac{1}{|\nabla u|^2}(\Hess\,u)(\nabla u,\nabla u)).
\end{equation}

The details are as follows. First note that the background metric in \eqref{e:levelMC} is
$$h_\rho^*\Phi^*g = g_0 + h_0 \rho^{-4} + O(\rho^{-5}).$$
The Christoffel symbols of this metric are given by
\begin{equation}\label{e:christ}
\Gamma_{ij}^k = \frac{1}{2}(\partial_i h_{0,jk} + \partial_{j}h_{0,ik} - \partial_k h_{0,ij})\rho^{-4} + O(\rho^{-5})
\end{equation}
in the standard Euclidean coordinates of $\R^4$ in a tubular neighborhood of $\S^3$. Then
\begin{equation*}
\begin{split}
\nabla_i u = 2x^i - 2h_{0,ij}x^j\rho^{-4} + O(\rho^{-5}) = 2x^i + O(\rho^{-5}),\\
\Hess_{ij}u = 2\delta_{ij} - (\partial_i h_{0,jk} + \partial_{j}h_{0,ik} - \partial_k h_{0,ij})x^k \rho^{-4} + O(\rho^{-5}),
\end{split}
\end{equation*}
where in the first line we have used the property $\partial_r \,\lrcorner\,h_0 = 0$ from \eqref{e:improve}. Thus,
\begin{equation*}
\begin{split}
|\nabla u|^2 = 4r^2 + 4h_{0,ij}x^ix^j\rho^{-4} + O(\rho^{-5}) = 4r^2 + O(\rho^{-5}),\\
\Delta u = 8 - 2(\Tr{g_0}{h_0} + (\cB_{g_0}h_0)_kx^k )\rho^{-4} + O(\rho^{-5}) = 8 + O(\rho^{-5}),\\
(\Hess\,u)(\nabla u,\nabla u) = 8r^2 - 4x^ix^jx^k(\partial_i h_{0,jk} + \partial_{j}h_{0,ik} - \partial_k h_{0,ij})\rho^{-4} + O(\rho^{-5}) = 8r^2 + O(\rho^{-5}),
\end{split}
\end{equation*}
thanks to all three of the properties of \eqref{e:improve} and by applying Euler's homogeneity relation to $h_0$ in the last line. Evaluating \eqref{e:levelMC} at $u = 1$, it is then clear that $H = 3 + O(\rho^{-5})$, as desired.

The upshot of all of this is that if we replace $\Phi$ by the diffeomorphism 
$$\rho x \mapsto \Phi((1+ f_\rho(x))\rho x) \;\;(\rho \gg 1, \, x \in \S^3/\Gamma),$$
then the leading term $h_0$ of the metric remains unchanged, and the remainder $h'$ satisfies the same estimates as before, but it now holds by construction  that $\Phi(\partial B_\rho(0)/\Gamma) = \Sigma_\rho$. \hfill $\Box$

\subsection{Proof of Theorem \ref{th:max-vol}} We now use the ALE diffeomorphism $\Phi$ provided by Theorem \ref{th:classterms}. One consequence of the property $\Tr{g_0}{h_0} = 0$ is that 
\begin{equation}\label{e:vol-error}
\Phi^*(d\vol_{g}) = d\vol_{\Phi^*g} = d\vol_{g_0} (1+ O(r^{-5})).
\end{equation}
This immediately implies that the function
$$g(\rho) := \vol_g(\Omega_\rho) - \vol_{g_0}(B_\rho(0)/\Gamma)$$
satisfies for all $\rho_2>\rho_1\gg1$ that
$$g(\rho_2)-g(\rho_1) = \int_{\rho_1}^{\rho_2} \Phi^*(d\vol_{g}) - d\vol_{g_0} = O(\rho_1^{-1}).$$
Thus, $g(\rho)$ has a finite limit as $\rho \to \infty$, which proves Theorem \ref{th:max-vol}(1), i.e., the existence of $\cV$. 

For Theorem \ref{th:max-vol}(2), we first prove that there exists a function $u$ on $M$ such that $\Delta_g u = 8$ and 
\begin{equation}\label{e:asympt-u}
\Phi^*u = r^2 + br^{-2} + O_\epsilon(r^{-3+\epsilon})
\end{equation}
for some $b \in \R$ and all $\epsilon \in (0,1)$. For this we require the precise expansion
$$\Delta_{\Phi^*g}r^2 = 8 - 2(\Tr{g_0}{h_0} + (\cB_{g_0}h_0)(r\partial_r) ) + O(r^{-5}) = 8 + O(r^{-5}),$$
which follows from the work in Section \ref{sec:proof-B} by pushing forward by the homotheties $h_\rho$. Following a standard pattern, we now extend $\Phi_*r^2$ to a smooth function $u_0$ on $M$ and let $f = 8-\Delta_g u_0$. Note that $\Phi^*f = O(r^{-5})$. It suffices to find a function $\bar{u}$ on $M$ such that $\Delta_g \bar{u}= f$ and 
$$\Phi^*\bar{u} =br^{-2} + O_\epsilon(r^{-3+\epsilon})$$ because then $u = u_0 + \bar{u}$ solves the original problem. The existence of $\bar{u}$ is a standard fact and can be proved in many ways. For example, by solving Dirichlet problems on larger and larger balls and using Moser iteration and weighted Sobolev inequalities, one first constructs a solution $\bar{u}$ such that $\Phi^*\bar{u} = O_\epsilon(r^{-2+\epsilon})$. Then it is clear that $\Delta_{g_0}(\Phi^*\bar{u}) = O(r^{-5})$, so standard properties of the Green's function on $\R^4$ imply that $\Phi^*\bar{u} = br^{-2} + O_\epsilon(r^{-3+\epsilon})$, as desired.

Let $\nu$ be the exterior unit normal to the domain $\Omega_\rho$ with respect to $g$. Integrating the equation $\Delta_g u = 8$ over $\Omega_\rho$, integrating by parts, and using \eqref{e:vol-error} we get
\begin{align*}
8 \vol_g(\Omega_\rho) &= \int_{\partial \Omega_\rho} du(\nu) (\nu \, \lrcorner\,d\vol_g) = \int_{\partial\Omega_\rho} du(\nu)(\nu\,\lrcorner\,\Phi_*d\vol_{g_0}) + O(\rho^{-1}).
\end{align*}
From now on we will ignore the map $\Phi$ for convenience. Then, using \eqref{e:asympt-u} in the first line,
\begin{equation*}
\begin{split}
du(\nu)|_{\partial\Omega_\rho} = 2\rho - 2b\rho^{-3}  + 2\rho( dr(\nu) - 1) + O_\epsilon(\rho^{-4+\epsilon}),\\
\nu\,\lrcorner\,d\vol_{g_0} = dr(\nu)(\partial_r \,\lrcorner\,d\vol_{g_0}).
\end{split}
\end{equation*}
A priori the error $dr(\nu) - 1$ is $O(\rho^{-4})$ but expanding the equation $g(\nu,\nu) = 1$ yields
$$dr(\nu) - 1 = - \frac{1}{2}h_0(\partial_r,\partial_r) + O(\rho^{-5}) =  O(\rho^{-5})$$
thanks to \eqref{e:improve}. Combining these computations, we get
$$8 \vol_g(\Omega_\rho) = 8\vol_{g_0}(B_\rho(0)/\Gamma) - 2b|\S^3/\Gamma | + O_\epsilon(\rho^{-1+\epsilon}),$$
which obviously implies that 
$b|\S^3/\Gamma| = -4\cV$.

We now reinterpret $b$ as the obstruction to $\nabla u$ being a conformal Killing field, and the proof will show that $b \geq 0$. Recall that $\nabla u$ being conformally Killing means that the trace-free Hessian $\Hess_0 u$  vanishes, and this is equivalent to $\nabla u$ generating a $1$-parameter group of conformal diffeomorphisms of $(M,g)$. In this case, the conformal factor is constant because $\Delta u = 8$, so $\nabla u$ actually generates a $1$-parameter group of homotheties, and by considering the sup norm of the curvature tensor one easily checks that this is equivalent to $(M,g)$ being isometric to $(\R^4/\Gamma, g_0)$ with $u = r^2$.

Using the general identity
$$\Delta \nabla u - \nabla \Delta u = {\rm Ric}\,\nabla u,$$
together with the fact that $u$ has constant Laplacian, we obtain that
\begin{align}\label{e:bochner-type}
\dvg(\Hess_0u)  = {\rm Ric}\,du = 0.
\end{align}
Integrating this against $d u$ over $\Omega_\rho$, we get
\begin{align*}
\int_{\partial\Omega_\rho} (\Hess_0 u)(\nu,\nabla u) (\nu\,\lrcorner\, d\vol) = \int_{\Omega_\rho} |\Hess_0 u|^2 + {\rm Ric}(\nabla u, \nabla u) \; d\vol \geq 0.
\end{align*}
The desired inequality $b \geq 0$, with equality if and only if $\Hess_0 u = 0$, will come out of an expansion of the boundary term. We begin by expanding $\Hess_0 u$. This will turn out to be $O(\rho^{-4})$ to leading order, so there is no need to also expand $\nu$, $\nabla u$, $\nu\,\lrcorner\,d\vol$, and we can simply replace these by their Euclidean approximations. For $\Hess_0 u$, using \eqref{e:christ} and \eqref{e:asympt-u} we have
\begin{align*}
\Hess_{0,ij}u &= \partial_i\partial_j u - \Gamma_{ij}^k\partial_k u - 2g_{ij}\\
&= -\frac{2b}{r^4}(\delta_{ij} - 4 \frac{x^ix^j}{r^2}) - (\partial_i h_{0,jk} + \partial_j h_{0,ik} - \partial_k h_{0,ij})x^k - 2h_{0,ij} + O_\epsilon(r^{-5+\epsilon}).
\end{align*}
This then tells us that, as $3$-forms on $\partial\Omega_\rho$,
\begin{align*}
&(\Hess_0 u)(\nu,\nabla u) (\nu\,\lrcorner\, d\vol) \\&= (12b\rho^{-3} - 2x^ix^jx^k(\partial_i h_{0,jk} + \partial_{j}h_{0,ik} - \partial_k h_{0,ij})\rho^{-1} - 4h_0(\partial_r,\partial_r)\rho + O_\epsilon(\rho^{-4+\epsilon}))(\partial_r \,\lrcorner\,d\vol_{g_0})\\
&= (12b\rho^{-3} + O_\epsilon(\rho^{-4+\epsilon}))(\partial_r\,\lrcorner\,d\vol_{g_0}),
\end{align*}
after applying Euler's relation to $h_0$ and using \eqref{e:improve} as in Section \ref{sec:proof-B}. 

Combining this with the preceding discussion, we obtain that $b \geq 0$, or equivalently that $\cV \leq 0$, with equality if and only if $(M,g)$ is isometric to $(\R^4/\Gamma,g_0)$ with $u = r^2$.\hfill $\Box$

\begin{rk}
The last part of the above proof can be interpreted as a reverse Bishop-Gromov type inequality. In fact, similar computations show that if $\Omega$ is a bounded domain in an $n$-manifold with $\Ric \geq 0$, with boundary mean curvature $\geq \frac{n-1}{\rho}$ for some $\rho > 0$, then $|\Omega| \leq \frac{\rho}{n}|\partial\Omega|$, with equality if and only if $\Omega$ is isometric to a ball of radius $\rho$ in $\R^n$. After completing our work on this paper, we discovered that this result already follows from Ros \cite[Thm 1]{Ros87}, who used a very similar method. Our computation replaces the use of Reilly's formula in \cite{Ros87} by \eqref{e:bochner-type} and \eqref{e:levelMC}.
\end{rk}

\subsection{Proof of Theorem \ref{th:max-vol-kron}} We now assume that $(M,g)$ is a Kronheimer gravitational instanton with period point $\zeta\in \mathfrak{h}\otimes \R^3$. As mentioned in Section \ref{sec:kron}, Kronheimer constructed a particular ALE diffeomorphism $\Phi: (\R^4 \setminus B_1(0))/\Gamma \to M\setminus K$ with respect to which $-(2\pi^2/|\Gamma|)r^6 h_0$ takes the form \eqref{e:kron}. Details can be found in \cite[Section 2]{Auv18}, including the remarkable property that
\begin{equation}\label{e:kronvol}
\vol_g(U_\tau) = \vol_{g_0}(B_\tau(0)/\Gamma)\;\,\text{for all}\;\,\tau\gg 1,
\end{equation}
where $U_\tau \subset M$ denotes the domain interior to $S_\tau := \Phi(\partial B_\tau(0)/\Gamma)$. (See \cite[Lemma 2.5]{Auv18} for this. Note that the map $F_\zeta$ of \cite{Auv18} is a global diffeomorphism from $M\setminus Z$ to $(\R^4 \setminus \{0\})/\Gamma$, where $Z \subset M$ is a finite union of compact $2$-dimensional surfaces.) By Lemma \ref{l:kron} and its proof, we have 
$$h_0 = h_0^+ + \cL_{X_1}g_0 + \cL_{X_2}g_0,$$
where, first of all, $h_0^+$ is a reduced Kronheimer term, so that
\begin{equation}
\partial_r \,\lrcorner\,h_0^+ = 0, \;\, \tr_{g_0}h_0^+ = 0, \;\, \dvg_{g_0}h_0^+ = 0;\label{e:improve+}
\end{equation}
second, the harmonic vector field $X_1$ is a scalar multiple of $\nabla_{g_0}(\frac{1}{r^2})$; and, third, the harmonic vector field $X_2$ is a linear combination of $I_j(\nabla_{g_0}(\frac{1}{r^2}))$ for $j = 1,2,3$. In fact, Remark \ref{rk:kron} tells us that 
\begin{equation}\label{e:sign}
X_1 = -\frac{|\tilde\zeta|^2}{12}\nabla_{g_0}(\frac{1}{r^2}), \;\text{where}\;\,\tilde\zeta := \left(\frac{|\Gamma|}{2\pi^2}\right)^{\frac{1}{2}}\zeta.
\end{equation}

We now replace Kronheimer's gauge $\Phi$ by $\Phi'= \Phi \circ \Phi_2 \circ \Phi_1$, where $\Phi_k$ denotes the time-$1$ flow of $-X_k$ for $k = 1,2$. In this new gauge, the leading term of the metric is $h_0^+$. By Section \ref{sec:proof-B}, because of \eqref{e:improve+}, the canonical CMC hypersurface $\Sigma_\rho$ is a normal graph of height $O(\rho^{-4})$ (rather than the expected $O(\rho^{-3})$) over the coordinate sphere $S_\rho' := \Phi'(\partial B_\rho(0)/\Gamma)$. Let $U_\rho' \subset M$ denote the domain interior to $S_\rho'$. Then it is easy to see that
\begin{align}\label{e:kronvol1}
\vol_g(\Omega_\rho) = \vol_g(U_\rho') + O(\rho^{-1}) = \vol_g(U_{\tau}) + O(\rho^{-1})\;\,\text{for}\;\, \tau^4 =\rho^4 - \frac{2}{3}|\tilde\zeta|^2.
\end{align}
Indeed, $S_\rho' = \Phi(\Phi_2(\Phi_1(\partial B_\rho(0)/\Gamma)))$ by definition, $\Phi_1(\partial B_\rho(0)/\Gamma) = \partial B_\tau(0)/\Gamma$ with $\tau$ in terms of $\rho$ as in \eqref{e:kronvol1} by solving a simple ODE (using \eqref{e:sign}), and $\Phi_2$ obviously preserves each sphere in $\R^4/\Gamma$. The desired formula for $\mathcal{V}(M,g)$ now follows from \eqref{e:kronvol1} and \eqref{e:kronvol}. \hfill $\Box$

\section{Killing fields on $4$-dimensional Ricci-flat ALE spaces}

The idea of the proof of Theorem \ref{th:max-vol} is to search for a conformal Killing vector field asymptotic to $2r\partial_r$. This is done by first producing the harmonic vector field $\nabla u$, where $u$ is asymptotic to $r^2$ with $\Delta u=8$, and then integrating by parts in the Bochner type formula $\dvg(\Hess_0 u) = \Ric\,du = 0$. Using a similar approach, one can give a criterion for the existence of a Killing field asymptotic to a given $\so(4)$ symmetry. The necessary computations are very long, and there does not seem to be an application in the spirit of Theorem \ref{th:max-vol}, so we will only briefly sketch this result.

Let $(M,g)$ be a $4$-dimensional Ricci-flat ALE space asymptotic to $\R^4/\Gamma$. Fix an arbitrary ALE diffeomorphism $\Phi$ as in \eqref{eq:1}. Then for all $X_0 \in \so(4)^\Gamma$ there exists a unique harmonic vector field $X$ on $M$ such that $(\Phi^{-1})_*X=X_0+Y_0+O_\epsilon(r^{-4+\epsilon})$ for all $\epsilon \in (0,1)$, where
$$\cL_{r\partial_r}Y_0 = -4Y_0, \;\, |Y_0|_{g_0} \sim r^{-3}.$$
An asymptotic expansion shows that $$\Delta_{g_0}Y_0 = 0.$$
The next observation is that the divergence of a harmonic vector field on a Ricci-flat manifold is a harmonic function, and it is clear that $\dvg\,X$ goes to zero at infinity, so $\dvg\,X = 0$ by the maximum principle. If the leading term $h_0$ of the metric with respect to $\Phi$ satisfies $\tr_{g_0}h_0 = 0$, then one can deduce from this by asymptotic expansion that $$\dvg_{g_0}Y_0 = 0.$$
Since $Y_0$ has to be of the form \eqref{e:foo}, it follows that for some $a \in \R$ and $Z_0 \in \so(4)$,
$$Y_0 = \frac{1}{r^4}(a r\partial_r + Z_0).$$

If $u$ is the unique solution to $\Delta u = 8$ with $u = r^2 + o(1)$ on $M$, and if the precise expansion \eqref{e:asympt-u} holds for $u$ (which we know is true if $\tr_{g_0}h_0 = 0$ and $\cB_{g_0}h_0 = 0$), then an integration by parts in the identity $0 = \langle \Delta X, \nabla u\rangle - \langle X, \Delta\nabla u\rangle$ shows after a long computation that $a = 0$. 

If $X_0, X_0'$ are in $\so(4)^\Gamma$ with harmonic extensions $X, X'$, and if $\partial_r \,\lrcorner\, h_0 = 0$, then another lengthy integration by parts in the identity $0 = \langle \Delta X, X'\rangle - \langle X, \Delta X'\rangle$ shows that $\langle X_0,Z_0' \rangle = \langle Z_0, X_0'\rangle$ with respect to the Killing form on $\so(4)$. Thus, the endomorphism $X_0 \mapsto Z_0$ of $\so(4)^\Gamma$ is selfadjoint.

Integration by parts in the Bochner type formula \eqref{e:bochner} shows that the selfadjoint endomorphism $X_0 \mapsto Z_0$ of $\so(4)^\Gamma$ is nonpositive, and that $Z_0 = 0$ if and only if $X$ is a Killing field. In this case, an asymptotic expansion of the equation $\cL_X g = 0$ immediately tells us that $\cL_{X_0}h_0 = 0$.

\appendix

\section{Computer code for the proof of Proposition \ref{p:classify}}
\label{app:code}

We ran the following script in Maxima 5.27.0 \cite{Max} to find the dimension of the kernel of the linear map $\mathbf{H}$ defined in \eqref{e:def_H}. This was one of the key technical steps of this paper.

The script represents $\mathbf{H}$ as a matrix $H$ with respect to natural bases of $\Sym^2\R^4 \otimes \Sym^2_0\R^4$ (with $10 \times 9 = 90$ elements) and of $\R^4 \otimes \Sym^3\R^4$ (with $4 \times 20 = 80$ elements). Our preferred basis of the domain consists of tensors $h \otimes q$, where $h$ runs over the standard basis of symmetric $4\times 4$ matrices (the loop over $m$ in the script) and $q$ runs over the first $9$ elements of the same basis made trace-free (the loop over $n$ in the script). To compute the ($9(m-1)+n$)-th column of $H$ according to \eqref{e:def_H}, we first convert $q$ into a harmonic quadratic polynomial \emph{qpol} in the variables $x_1,x_2,x_3,x_4$. Then for all $k \in \{1,2,3,4\}$, the $k$-th component of the $1$-form $r^8\cB_{g_0}(\frac{\textit{qpol}}{r^6}h)$ is a cubic polynomial \emph{cubic} in the same variables, and for $\ell \in \{1,2,3,\ldots,20\}$ we store the coefficient of the $\ell$-th cubic basis monomial in \emph{cubic} in row number $20(k-1)+\ell$ (and column number $9(m-1)+n$) of the matrix $H$.

The procedure to read out the coefficients of the basis monomials in \emph{cubic} is a little convoluted due to the following problem. If $\textit{cubic} = 10x_1x_3x_4 - 8x_2x_3^2$, then $\textit{coeff}(\textit{coeff}(\textit{coeff}(\textit{cubic},x_1),x_3),x_4)$ returns $10$, but $\textit{coeff}(\textit{coeff}(\textit{coeff}(\textit{cubic},x_2),x_3),x_3)$ returns $0$ by default because $-8x_3^2$ is nonlinear in $x_3$, so we must instead use $\textit{coeff}(\textit{coeff}(\textit{cubic},x_2),x_3^2)$ to get the coefficient $-8$.

\vskip3mm

\noindent load(linearalgebra)\$

\noindent load(functs)\$

\noindent X: matrix([x1,x2,x3,x4])\$

\noindent R: x1\^{}2+x2\^{}2+x3\^{}2+x4\^{}2\$

\noindent A: matrix([1,2],[1,3],[1,4],[2,3],[2,4],[3,4],[1,1],[2,2],[3,3],[4,4])\$

\noindent B: matrix([1,1,1],[1,1,2],[1,1,3],[1,1,4],[1,2,2],[1,2,3],[1,2,4],[1,3,3],[1,3,4],[1,4,4],

\noindent \hskip17.8mm [2,2,2],[2,2,3],[2,2,4],[2,3,3],[2,3,4],[2,4,4],[3,3,3],[3,3,4],[3,4,4],[4,4,4])\$

\noindent H: zeromatrix(80,90)\$

\noindent for m: 1 thru 10 do(

\noindent \hskip7mm h: zeromatrix(4,4),

\noindent \hskip7mm h[A[m,1],A[m,2]]: 1,

\noindent \hskip7mm h[A[m,2],A[m,1]]: 1,

\noindent \hskip7mm for n: 1 thru 9 do(

\noindent \hskip14mm qpre: zeromatrix(4,4),

\noindent \hskip14mm qpre[A[n,1],A[n,2]]: 1,

\noindent \hskip14mm qpre[A[n,2],A[n,1]]: 1,

\noindent \hskip14mm q: qpre-(1/4)*tracematrix(qpre)*ident(4),

\noindent \hskip14mm qpol: sum(sum(q[u,v]*X[1,u]*X[1,v],u,1,4),v,1,4),

\noindent \hskip14mm for k: 1 thru 4 do(

\noindent \hskip21mm cubic: expand(ratsimp(R\^{}4*(sum(h[p,k]*diff(qpol/R\^{}3,X[1,p]),p,1,4)

\noindent \hskip32.5mm  -(1/2)*tracematrix(h)*diff(qpol/R\^{}3,X[1,k])))),
			
\noindent \hskip21mm  for l: 1 thru 20 do(

\noindent \hskip28mm if (B[l,1]$<$B[l,2])

\noindent \hskip35mm then

\noindent \hskip42mm (if (B[l,2]$<$B[l,3])

\noindent \hskip49mm then 

\noindent \hskip56mm (H[20*(k-1)+l,9*(m-1)+n]:

\noindent \hskip58mm coeff(coeff(coeff(cubic,X[1,B[l,1]]),X[1,B[l,2]]),X[1,B[l,3]]))

\noindent \hskip49mm else

\noindent \hskip56mm (H[20*(k-1)+l,9*(m-1)+n]:

\noindent \hskip58mm coeff(coeff(cubic,X[1,B[l,1]]),X[1,B[l,2]]\^{}2))

\noindent \hskip42mm )

\noindent \hskip35mm else

\noindent \hskip42mm (if (B[l,2]$<$B[l,3])

\noindent \hskip49mm then 

\noindent \hskip56mm (H[20*(k-1)+l,9*(m-1)+n]:

\noindent \hskip58mm coeff(coeff(cubic,X[1,B[l,1]]\^{}2),X[1,B[l,3]]))

\noindent \hskip49mm else

\noindent \hskip56mm (H[20*(k-1)+l,9*(m-1)+n]:

\noindent \hskip58mm coeff(cubic,X[1,B[l,1]]\^{}3))
							
\noindent \hskip42mm )

\noindent \hskip21mm )

\noindent \hskip14mm )

\noindent \hskip7mm )

\noindent )\$

\noindent nullity(H);

\bibliographystyle{amsplain}
\bibliography{maxvol7}

\end{document}